\documentclass[a4paper,10pt]{amsart}

\usepackage{graphicx}

\usepackage{cite}
\usepackage{todonotes}
\usepackage{hyperref}

\hypersetup{%
pdftitle={},
pdfsubject={Mathematics},
pdfauthor={Manfred Madritsch, Jo\"el Rivat, Robert Tichy},
pdfkeywords={}
hyperindex=true,plainpages=false}

\usepackage{a4wide}

\usepackage{amsmath}
\usepackage{amsfonts}
\usepackage{amssymb}
\usepackage{amsthm}
\usepackage{mathtools}

\usepackage{stmaryrd}

\newtheorem{lem}{Lemma}[section]
\newtheorem{thm}[lem]{Theorem}

\numberwithin{equation}{section}

\newtheorem*{cor*}{Corollary}
\newtheorem*{thm*}{Theorem}

\theoremstyle{definition}
\newtheorem{defi}{Definition}[section]

\theoremstyle{remark}

\allowdisplaybreaks[3]

\newcommand{\N}{\mathbb{N}}
\newcommand{\Z}{\mathbb{Z}}
\newcommand{\Q}{\mathbb{Q}}
\newcommand{\R}{\mathbb{R}}

\renewcommand{\lvert}{\left\vert}
\renewcommand{\rvert}{\right\vert}

\newcommand{\abs}[1]{\left| #1 \right|}
\newcommand{\norm}[1]{\left\| #1 \right\|}
\newcommand{\floor}[1]{\left\lfloor #1 \right\rfloor}

\title[On Finite Pseudorandom Binary Sequences]{On Finite Pseudorandom Binary
  Sequences: \\Functions from a Hardy field}

\author[M. G. Madritsch]{Manfred G. Madritsch}
\address[M. G. Madritsch]{
  Université de Lorraine, CNRS, IECL, F-54000 Nancy, France}
\email{manfred.madritsch@univ-lorraine.fr}

\author[J. Rivat]{Jo\"el Rivat}
\address[J. Rivat]{
  Universit\'e d'Aix-Marseille\\
  Institut Universitaire de France\\
  Institut de Math\'ematiques de Marseille\\
  CNRS UMR 7373\\
  163, avenue de Luminy, Case 907\\
  13288 MARSEILLE Cedex 9, France.}
\email{joel.rivat@univ-amu.fr}

\author[R. F. Tichy]{Robert F. Tichy}
\address[R. F. Tichy]{
  Graz University of Technology, Institute of Analysis and Number
  Theory, 8010~Graz, Austria}
\email{tichy@tugraz.at}



\date{\today}

\begin{document}

\begin{abstract}
  We provide a construction of binary pseudorandom sequences based on
  Hardy fields $\mathcal{H}$ as considered by Boshernitzan.  In
  particular we give upper bounds for the well distribution measure
  and the correlation measure defined by Mauduit and Sárközy. Finally
  we show that the correlation measure of order $s$ is small only if
  $s$ is small compared to the ``growth exponent'' of $\mathcal{H}$.
\end{abstract}

\maketitle

\section{Introduction}
A pseudorandom binary sequences is a sequence over $\{-1,+1\}$ which
is generated by a ``simple algorithm'' such that it behaves
randomly. By behaving randomly we mean that the sequences is
indistinguishable from a real random sequence.

The simple algorithm is based on the following construction. First we
define $\chi(x)$ by
\begin{gather}\label{mani:Xi}
\chi(x)=\begin{cases}
+1 &\text{for $0\leq\{x\}<1/2$,}\\
-1 &\text{for $1/2\leq\{x\}<1$,}
\end{cases}
\end{gather}
where $\{\cdot\}$ denotes the fractional part. Then for a given
function $f$ and for integers $N\geq1$ we define the associated binary
sequence of length $N$ by $E_N=E_N(f)=\{e_1,\ldots,e_N\}$, where
$e_n=\chi(f(n))$.

In order to show that this construction yields in fact a pseudorandom
sequence we introduce quantities, measuring the quality of its
randomness. In the present paper we want to focus on two measures,
which were introduced by Mauduit and Sárközy in Part I
\cite{mauduit_sarkoezy1997:finite_pseudorandom_binary}:
\begin{itemize}
\item \textbf{Well-distribution measure:} For a real random sequence
  we expect that we have no arithmetic progression with many equal
  values. In particular, for a binary sequence $E_N$, $b\in\Z$ and
  $a,M\in\N$ with $1\leq a+b\leq aM+b\leq N$ we set
\[
  U(E_N,M,a,b):=\sum_{n=1}^M e_{an+b}.
\]
Then the \textit{well-distribution measure} of $E_N$ is defined as
\[
  W(E_N)
  =\max_{a,b,M}\lvert U(E_N,M,a,b)\rvert
  =\max_{a,b,M}\lvert\sum_{n=1}^M e_{an+b}\rvert,
\]
where the maximum is taken over all $a$, $b$ and $M$ such that $b\in\Z$,
$a,M\in\N$ with $1\leq a+b\leq aM+b\leq N$.
\item \textbf{Correlation measure:} Another property of a real random
  sequence is the lack of correlation between its elements. To this
  end let $s\geq2$ be a positive integer and $E_N$ be a binary
  sequence. For a given positive integer $M$ and a given correlation
  vector $\mathbf{d}=(d_1,\ldots,d_s)\in\N^s$ such that
  $0\leq d_1<d_2<\cdots< d_s\leq N-M$ we set
\[
  V(E_N,M,\mathbf{d})=\sum_{n=1}^Me_{n+d_1}e_{n+d_2}\cdots e_{n+d_s}.
\]
Note that we have supposed that $s\geq2$ since the case $s=1$
corresponds to $U(E_N,M+d_1,1,d_1)$.

Then the \textit{correlation measure of order $s$ of $E_N$} is defined as
\[
  C_s(E_N)=\max_{M,\mathbf{d}}\lvert V(E_N,M,\mathbf{d})\rvert
  =\max_{M,\mathbf{d}}\lvert \sum_{n=1}^Me_{n+d_1}e_{n+d_2}\cdots e_{n+d_s}\rvert,
\]
where the maximum is taken over all $M$ and $\mathbf{d}\in\N^s$ such that
$0\leq d_1<d_2<\ldots<d_s$ and $M+d_s\leq N$.
\end{itemize}

Both measures are in connection with the distribution modulo $1$ of the sequence
$(f(n))_{n\geq1}$. For example consider the well-distribution measure for a
fixed $E_N$, $a$, $b$ and $x$.
We introduce the indicator of a condition $\mathcal{C}$:
$\mathbf{1}_{\mathcal{C}} = 1$ if condition $\mathcal{C}$ is
satisfied and $\mathbf{1}_{\mathcal{C}} = 0$ otherwise.
Then by the definition of $\chi$ we have
\begin{align*}
  \abs{\sum_{an+b\leq x}e_{an+b}}
  &\leq\abs{\sum_{an+b\leq x}\left(\mathbf{1}_{\{f(an+b)\}<\tfrac12}-\frac12\right)}
  +\abs{\sum_{an+b\leq x}\left(\mathbf{1}_{\{f(an+b)\}\geq\tfrac12}-\frac12\right)}\\
  &\leq 2MD_M(x_1,\ldots,x_M),
\end{align*}
where $M=\left\lfloor\frac{t-b}{a}\right\rfloor$, $x_n=f(an+b)$ for $1\leq n\leq
M$ and $D_M$ is the discrepancy defined in Definition~\ref{def:discrepancy}
below. Therefore a discrepancy estimate for the sequence $(f(an+b))_{n\geq1}$
gives rise to an estimate for the well-distribution measure. 

In the present paper we concentrate on polynomial-like sequences. On the one
hand in two papers Mauduit and Sárközy \cite{
mauduit_sarkoezy2000:finite_pseudorandom_binary,mauduit_sarkoezy2000:finite_pseudorandom_binary2}
investigated monomials $f(x)=\alpha x^k$ with integer $k\geq1$ and irrational
$\alpha\in\R\setminus \Q$ with bounded partial coefficients. This means that
there exists $K\geq1$ such that
\begin{gather}\label{eq:bounded_coefficients}
  \alpha=[a_0;a_1,a_2,\ldots],\quad a_i\leq K\text{ for }i\geq1.
\end{gather}
In particular, they showed the following theorem.
\begin{thm*}[{\cite[Theorems 1 and
  2]{mauduit_sarkoezy2000:finite_pseudorandom_binary2}}] Let $k,\ell$ be
  positive integers such that $k\geq3$ and $2\ell+1\leq k$. Furthermore let
  $\alpha$ satisfy \eqref{eq:bounded_coefficients} with some $K\in\R$. Then
  there exists $\sigma_k$ depending on $k$ only such that for all $\varepsilon$,
  there is $N_0=N_0(K,k,\varepsilon)$ such that if $N>N_0$, then
  \[ W(E_N)<N^{1-1/\sigma_k+\varepsilon}
  \quad\text{and}\quad
  C_\ell(E_N)<N^{1-1/\sigma_k+\varepsilon},\]
  where $E_N=E_N(\alpha n^k)$.
\end{thm*}
Similar results for $k=1$ and $k=2$ can be found in Mauduit and
Sárközy \cite{mauduit_sarkoezy2000:finite_pseudorandom_binary}.

On the other hand Mauduit, Rivat and Sárközy
\cite{mauduit_rivat_sarkoezy2002:pseudo_random_properties} 
considered polynomial like functions, i.e. functions $f(x)=x^c$ with non-integer
$c>1$. They were able to show the following.
\begin{thm*}[{\cite[Theorem 1 and 2]{mauduit_rivat_sarkoezy2002:pseudo_random_properties}}]
  Let $c>1$ with $c\not\in\Z$ and set $R=\lceil c\rceil$.
  \begin{enumerate}
  \item For $0\leq b\leq a\leq (x-b)^{1-c/R}$ and $x\to\infty$, we have
    \[\abs{\sum_{an+b\leq x}e_{an+b}}\ll a^{-1+R/(2^R-1)}x^{1-(R-c)/(2^R-1)}.\]
  \item For $1\leq d\leq N^{1-2(R-c)/(2^R-1)}$ and $N\to\infty$, we have
    \[\abs{\sum_{n\leq N}e_ne_{n+d}}
    \ll cN^{1-(R-c)/(2^R-1)}\log^2N+\frac{N^{2-c}}{d}.\]
  \end{enumerate}
\end{thm*}

\section{The results}

The aim of the the present paper is to extend the considerations for the
function $f(x)=x^c$ with $c>1$ by
Mauduit, Rivat and Sárközy
\cite{mauduit_rivat_sarkoezy2002:pseudo_random_properties} to general
subpolynomial functions. This is motivated by the above mentioned connection
with discrepancy estimates and recent results by Bergelson \textit{et al.}
\cite{bergelson_kolesnik_madritsch+2014:uniform_distribution_prime}, Bergelson
and Richter \cite{bergelson_richter2017:density_coprime_tuples} and Bergelson,
Kolesnik and Sun
\cite{bergelson_kolesnik_son2019:uniform_distribution_subpolynomial} in that
direction.

We call two functions $f$ and $g$ equivalent (at $+\infty$) if they agree
eventually, \textit{i.e.} if there exists $x_0>0$ such that $f(x)=g(x)$ for all
$x>x_0$. The corresponding equivalence classes are often called germs of
functions. Let $B$ be the collection of equivalence classes. Then a Hardy field
is a subfield of the ring $(B,+,\cdot)$ that is closed under differentiation. We
denote by $\mathcal{H}$ the union of all Hardy fields. Without loss of
generality we may suppose that each representative $f\in\mathcal{H}$ is defined
in $[1,\infty[$.

This notion goes back to Hardy's consideration of logarithmico-exponential
functions
(\cite{hardy1971:orders_infinity, hardy1912:properties_logarithmico_exponential})
and some well known examples are the following functions:
\[x^c\,(c\in\R),\quad
\log(x),\quad
\exp(x),\quad
\Gamma(x),\quad
\zeta(x),\quad
\mathrm{Li}(x),\quad
\sin\left(\tfrac1x\right),\quad
\text{etc.}\]

The elements of $\mathcal{H}$ satisfy some properties, which are relevant to our
study. First we note that every $f\in\mathcal{H}$ has eventually constant sign
(because of the multiplicative inverse). This implies that every $f$ is
eventually monotone, since $f'\in\mathcal{H}$ (closed under differentiation) and
$f'(x)$ has eventually constant sign. Therefore $\lim_{x\to\infty}f(x)$ exists
(possible infinite).

For two functions $f,g\in \mathcal{H}$ we write $f(x)\prec g(x)$ if
$\frac{f(x)}{g(x)}\to0$ for $x\to\infty$. Then in view of 
Boshernitzan \cite{boshernitzan1994:uniform_distribution_and},
for a non negative integer $\ell$, we call
$f\in\mathcal{H}$ of type $x^{\ell+}$
if $x^\ell\prec f(x)\prec x^{\ell+1}$,
provided that the ``growth exponent''
$\beta(f):=\inf\{c\in[0,\infty[\colon f(x)\prec x^c\} > 0$.
Note that this condition garanties that $f$ is not a polynomial.

\begin{thm}\label{thm:well-distribution} Let $\ell\geq0$ be an integer. If
  $f\in\mathcal{H}$ is of type $x^{\ell+}$, then there exists
  $\sigma=\sigma(f)>0$ depending only on $f$ such that for all $\varepsilon>0$
  we have
  \[W(E_N)\ll N^{1-\sigma+\varepsilon},\]
  where the implied constant may depend on $f$ and $\varepsilon$.
\end{thm}

Following the proof one can explicitely calculate $\sigma(f)$. However, we do not
claim that our choice is optimal and thus we omit these calculations in the
present paper.


\begin{thm}\label{thm:correlation_measure_bound}
  If $f\in\mathcal{B}$ is of type $x^{\ell+}$ with $\ell\geq1$, then for
  $\beta=\beta(f)$, $2\leq s< \beta+1$ and $\varepsilon>0$, there exists
  $\eta>0$ depending on $\beta$ only such that
  \[
    C_s(E_N)\ll N^{1-\eta+\varepsilon},
  \]
  where the implied constant may depend on $f$ and $\varepsilon$.
\end{thm}

We note that following our proof one can explicitely calculate $\eta$ for given
$\beta$. However, since our choices are not optimal, we skip the details here.

Our final result is a counter example saying that the bound $s\leq \ell$ is sharp.

\begin{thm}
  Let $0<c<1$. Then for $f(x)=x^c$ we have
  \[C_2(E_N)\gg N.\]
\end{thm}

The structure of the present paper is as follows. In Section
\ref{section:useful-estimates} we lay out our toolbox full of
estimates for the discrepancy, the derivative and exponential sums,
which we need for the proofs.

\section{Useful estimates}\label{section:useful-estimates}

In this section we want to collect some estimates, which will be of importance
in the proofs in the following sections.

\subsection{Discrepancy estimates}

As already mentioned in the introduction the pseudo-random measures are linked
to the discrepancy of the underlying sequence.

\begin{defi}\label{def:discrepancy}
  Let $N\geq1$ be a positive integer and let $\mathbf{x}_1,\ldots,\mathbf{x}_N$ be a finite
  sequence of points in $\R^s$. Then the discrepancy of $\mathbf{x}_1,\ldots,\mathbf{x}_N$ is
  defined by
  \[D_N(\mathbf{x}_1,\ldots,\mathbf{x}_N)
  =\sup_{I_1,\ldots,I_s}\abs{\frac1N
  \sum_{i=1}^N\mathbf{1}_{\{\mathbf{x}_i\}\in I_1\times\cdots\times I_s}-
  \mu(I_1)\ldots\mu(I_s)},\]
  where the supremum is taken over all intervals $I_1,\ldots,I_s$ contained in
  $[0,1]$, $\mu(I)$ stands for the length of the interval $I$, and \textit{i.e.}$\{\mathbf{x}\}=(\{x_1\},\ldots,\{x_s\})\in[0,1]^s$.
\end{defi}

The first estimate deals with the case of the discrepancy of a one-dimensional
sequence. This will be of importance in the proof of the upper bound for the
well-distribution measure.

\begin{lem}[\textsc{Erd\H{o}s-Turán}]
  \label{lem:erdos-turan}
  For any integers $N>0$, $H>0$, and any real numbers $x_1,\ldots,x_N$
  we have
  \[D_N(x_1,\ldots,x_N)\ll\frac1{H+1}+\sum_{h=1}^H\frac1h\abs{\frac1N\sum_{n=1}^N e(hx_n)}.\]
\end{lem}

\begin{proof}
  For a proof with the implied constant equal to $1$,
  see \cite[Lemma~1]{mauduit_rivat_sarkoezy2002:pseudo_random_properties}
  (and see also \cite{rivat-tenenbaum-2005}).
\end{proof}

For the correlation measure we need an estimate for the multidimensional
discrepancy, which is provided by the following lemma.

\begin{lem}[\textsc{Koksma-Sz\H{u}sz}]
  \label{lem:koksma-szusz}
  Let $s>0$ be an integer. For $\mathbf{h}=(h_1,\ldots,h_s)\in\Z^s$, write
  \[
    \varphi(\mathbf{h})=\max_{j=1,\ldots,s}\abs{h_j},\quad
    \mathbf{r}(\mathbf{h})=\prod_{j=1}^s\max(\abs{h_j},1).
  \]
  Let $\mathbf{x}_1,\ldots,\mathbf{x}_N$ be a finite sequence of points in
  $\R^s$. For any integer $H>0$ we have
  \[
    D_N(\mathbf{x}_1,\ldots,\mathbf{x}_N)
    \ll_s\frac1H+\frac1N\sum_{0<\varphi(\mathbf{h})\leq H}
    \frac1{\mathbf{r}(\mathbf{h})}\abs{\sum_{n=1}^N
    e(\mathbf{h}\cdot \mathbf{x}_n)}.
  \]
\end{lem}

\begin{proof}
  This is \cite[Lemma 3]{mauduit_rivat_sarkoezy2002:pseudo_random_properties}.
\end{proof}

For the counterexample we need a lower bound. Therefore we need to approximate
the functions $\chi$ by trigonometric sums.

\begin{lem}\label{lem:vaaler}
  For any  integer $H\geq 1$ 
  there exist real valued
  trigonometric polynomials $A_{H}(x)$ and $B_{H}(x)$ 
  such that for any $x\in\R$
  \begin{equation}\label{eq:vaaler-approximation}
    \left| \chi(x) - A_{H}(x) \right|
    \leq
    B_{H}(x),
  \end{equation}
  where
  \begin{align}
    \label{eq:definition-A}
    A_{H}(x)
    =2
    \sum_{1\leq\abs{h}\leq H}a_h(H) e(hx)
    \\
    \label{eq:definition-B}
    B_{H}(x)
    =
    b_0(H) + 2\sum_{h=1}^H  b_h(H) e(hx),
  \end{align}
  with coefficients $a_h(H)$ and $b_h(H)$ defined by
  \begin{equation}\label{eq:vaaler-coef-a}
    a_h(H) = e\left(-\tfrac{h}{2}\right)\left(
    \tfrac{\sin \frac{\pi h}{2} }{\pi h}
    \left( 
      \pi \tfrac{\abs{h}}{H+1} \left(1-\tfrac{\abs{h}}{H+1}\right) 
      \cot \pi \tfrac{\abs{h}}{H+1}
      +
      \tfrac{\abs{h}}{H+1}\right)
    \right)
  \end{equation}  
  \begin{equation}\label{eq:vaaler-coef-b}
    b_h(H) = e\left(-\tfrac{h}{2}\right)
    \tfrac{1}{H+1} \left(1-\tfrac{\abs{h}}{H+1}\right) 
    \cos\frac{\pi h}{2}
    .
  \end{equation}
\end{lem}

Note that for $1\leq h\leq H$ we have
\begin{gather}
  \abs{a_h(H)}=\abs{a_{-h}(H)}\asymp \frac1h
  \quad\text{and}\quad
  \abs{b_h(H)}=\abs{b_{-h}(H)}\asymp \frac1{H}.
\end{gather}

\begin{proof}
  Apply Lemma 1 of \cite{mauduit_rivat2015:prime_numbers_along} with
  $\chi= 2\chi_{1/2}-1$.
\end{proof}

\subsection{Estimates for derivatives in Hardy fields}

Now we turn our attention to estimates of the growth of subpolynomial functions
and their derivatives.






\begin{lem}%
  \label{lem:higher_derivative_order}
  Let $f\in\mathcal{H}$ have type $x^{\ell+}$ with $\ell\geq0$. Then
  for all $j\in\N$ we have
  \[\frac{f(x)}{x^j\log^2(x)}\prec \abs{f^{(j)}(x)} \ll\frac{f(x)}{x^j}.\]
\end{lem}

\begin{proof}
  This is \cite[Corollary 2.3]{frantzikinakis2009:equidistribution_sparse_sequences}.
\end{proof}



\subsection{Exponential sum estimates}
In this section we estimate the occurring exponential sums. To this end we state
two classical theorems all of which can be found in Chapter 2 of the book of
Graham and Kolesnik \cite{graham_kolesnik1991:van_der_corputs}.

We start with the Kusmin-Landau inequality for exponential sums (cf.
\cite[Theorem 2.1]{graham_kolesnik1991:van_der_corputs}), which we
will need for estimates based on the first derivative.

\begin{lem}[\textsc{Kusmin-Landau}]
  \label{lem:Kusmin_Landau}
  If $g$ is continuously differentiable, $g'$ is monotone, and
  $\norm{g'}\geq\lambda>0$ on $I$
  (where $\norm{t} = \min_{n\in\Z} \abs{t-n}$) then
  \[
    \abs{\sum_{n\in I}e(g(n))} \ll \lambda^{-1}.
  \]
\end{lem}




The following lemma is an improved version of \cite[Lemma~2.5]{bergelson_kolesnik_madritsch+2014:uniform_distribution_prime} based on
\cite[Theorem 2.9]{graham_kolesnik1991:van_der_corputs}.

\begin{lem}
  [{\cite[Lemma 2.10]
  {bergelson_kolesnik_son2019:uniform_distribution_subpolynomial}}]
  \label{lem:BKS}
  Let $r\geq 2$ and denote $R=2^{r-1}$. Let $f(x)$ be an $r$-times continuously
  differentiable real function on $I=]X_1,X_1+X]\subset]X_1,2X_1]$, where
  $X,X_1\in\N$. Assume further that $f^{(r)}(x)$ is monotone on $I$ and
  $\lambda_r \leq \abs{f^{(r)}(x)}\leq \alpha_r\lambda_r$ for some
  $\lambda_r ,\alpha_r>0$. Then we have
  \begin{align*}
    \abs{\sum_{n\in I}e(f(n))}
    \ll X\left[(\alpha_r\lambda_r)^{\frac1{2R-2}}+(\lambda_r X^{r})^{-\frac1R}
    (\log X)^{\frac{r-1}R}+\left(\frac{\alpha_r \log^{r-1}X}{X}\right)^{\frac1R}\right],
  \end{align*}
  where the implied constant depends on $r$ only.
\end{lem}




The final estimate is a lower bound for an exponential sum, which occurs in the
counterexample.

\begin{lem}\label{lem:robert}
  Let $f:\N \to \R^+$ be a monotonically
  increasing function. Furthermore let 
  \[
    f\left(\left\lfloor \frac{N}{4}\right\rfloor\right)\geq 4\pi\abs{h}
  \]
  with $h\in\Z$ and $N\in\N$. Then
  \[
    \abs{\sum_{n=1}^{N}e\left(\frac{h}{f(n)}\right)}\geq \frac{N}{8}.
  \]
\end{lem}

\begin{proof}
  Let $1\leq M\leq N$ be a positive integer we will choose in an instant. Then
  \begin{align*}
    \abs{\sum_{n=1}^{N}e\left(\frac{h}{f(n)}\right)}
    &=\abs{\sum_{n=M+1}^{N}e\left(\frac{h}{f(n)}\right)-\sum_{n=1}^{M}\left(-e\left(\frac{h}{f(n)}\right)\right)}\\
    &\geq\abs{\sum_{n=M+1}^{N}e\left(\frac{h}{f(n)}\right)} - M\\
    &=\abs{\sum_{n=M+1}^N1-\sum_{n=M+1}^N\left(1-e\left(\frac{h}{f(n)}\right)\right)} - M\\
    &\geq(N-M)-\sum_{n=M+1}^N\abs{e\left(\frac{h}{f(n)}\right)-1} - M\\
    &\geq(N-M)-2\pi\abs{h}\sum_{n=M+1}^N\frac{1}{f(n)} - M\\
    &\geq(N-M)-2\pi\abs{h}\frac{N-M}{f(M)} - M.
  \end{align*}
  By choosing $M=\lfloor N/4\rfloor$ we obtain that the last line is
  \begin{align*}
    &\geq N-M-\frac{N-M}{2}-M=\frac{N-M}2-M\\
    &\geq\frac{N-N/4}{2}-\frac{N}4=\frac{N}8.
  \end{align*}
\end{proof}

\section{Well-distribution}

Let $f\in\mathcal{H}$ be a function of type $x^{\ell+}$ with $\ell\geq0$.
Without loss of generality we may assume that $f$ is eventually increasing.
The idea is that we first show an estimate of the sum
\[\abs{\sum_{an+b\leq x}e_{an+b}}\]
for an arbitrary $b\in\Z$ and $a,x\in\N$ such that $1\leq a+b\leq ax+b\leq N$
and $a$ is ``small''. Then we vary over all possible $a$, $b$ and $x$ in that range.
Finally we show that the sum stays small if $a$ is ``large''. 

For given $a$, $b$ and $x$ we set $M=\left\lfloor\frac{x-b}{a}\right\rfloor$ for
short. Thus by the definition of $e_n$ and $\chi$ we get
\[
  \sum_{n\leq M}e_{an+b} =\sum_{n\leq M}
    \left(\mathbf{1}_{\{f(an+b)\}<\tfrac12}-\frac12\right)
  -\sum_{n\leq M}
    \left(\mathbf{1}_{\{f(an+b)\}\geq\tfrac12}-\frac12\right).
\]
Furthermore by the definition of the discrepancy we obtain
\[\abs{\sum_{n\leq M}e_{an+b}}\leq 2M D_M(x_1,\ldots,x_M),\]
where we wrote $x_n=f(an+b)$ for short.

By Lemma~\ref{lem:erdos-turan}, for any integer $H\geq 1$ we have that
\begin{gather}\label{eq:e_well-distribution}
  \abs{\sum_{n=1}^{M} e_{an+b}}
  \leq
  \frac{2M}{H+1}+2\sum_{h=1}^{H}
  \frac1h
  \abs{\sum_{n=1}^{M} e(hf(an+b))}.
\end{gather}

Next we consider the inner sum.
We split the summation over $n$ into $Q$ (to be chosen later)
intervals of the form $]M/2^q,M/2^{q-1}]$:
\begin{displaymath}
  \abs{\sum_{n=1}^{M} e(hf(an+b))}
  \leq
  \abs{\sum_{n\leq M/2^{Q+1}} e(hf(an+b))}
  +
  \sum_{0\leq q \leq Q}
  \abs{\sum_{M/2^{q+1}<n\leq M/2^{q}} e(hf(an+b))}
  .
\end{displaymath}
We majorize trivially the first sum.
This will produce in \eqref{eq:e_well-distribution} an error term
\begin{math}
  \ll \frac{M}{2^Q} \log H
  .
\end{math}

Now we focus on the exponential sum
\begin{displaymath}
    \sum_{P_q<n\leq 2P_q}e\left(g(n)\right)
\end{displaymath}
with $g(x)=hf(ax+b)$ and $P_q=M2^{-q-1}$.

Since $g^{(j)}(n)=ha^j f^{(j)}(an+b)$,
by Lemma \ref{lem:higher_derivative_order}
we get for $j\geq0$
\[\frac{\abs{h}f(an+b)}{n^j\log^2(an)}
\ll \abs{g^{(j)}(n)}
\ll \frac{\abs{h}f(an+b)}{n^j}.\]

Because of $\beta=\beta(f):=\inf\{c\in[0,\infty[\colon f(x)\prec x^c\}$
and $r=\left\lceil \beta+\tfrac12\right\rceil$,
for arbitrary $\varepsilon$ with $0<\varepsilon<1/8$ we have
\[
  x^{\beta-\varepsilon}\ll \abs{f(x)}\ll x^{\beta+\varepsilon},
\]
where the implied constant may depend on $\varepsilon$.

Since $f$ is assumed to be eventually increasing, we have
$f(aP_q+b)\ll f(2aP_q+b)$. Putting everything together we get
for $r>\beta$ and $n\in]P_q,2P_q]$ that

\[\lambda_r \ll\abs{g^{(r)}(n)}\ll\alpha_r\lambda_r\]
with
\begin{gather}\label{eq:lambda_and_alpha}
  \lambda_r=\abs{h}(aP_q)^{\beta-\varepsilon}P_q^{-r}
  \quad\text{and}\quad
  \alpha_r=(aP_q)^{2\varepsilon}.
\end{gather}

Next we distinguish three cases
$0 < \beta\leq \tfrac12$, $\tfrac12<\beta<1$ and
$1\leq \beta$.

\begin{lem}\label{lem:well-distribution-small-beta}
  Suppose that $\beta(f)\leq\tfrac12$, then, taking
  $\sigma = \min\left(\frac15,\beta\right)$,
  for arbitrary $\varepsilon>0$ we have
  \[
    \abs{\sum_{n\leq M}e_{an+b}}
    \ll
    M^{1-\sigma+\varepsilon}
    .
  \]
\end{lem}

\begin{proof}
  In this case $r=1$ and by \eqref{eq:lambda_and_alpha} we have that
  \[
    \lambda_{1}^{-1}=\frac{P_q}{\abs{h}(aP_q)^{\beta-\varepsilon}}
    ,
  \]
  and since $\beta-1+\varepsilon < 0$ we have
  \begin{displaymath}
    \alpha_1 \lambda_1
    =
    \abs{h}(aP_q)^{\beta+\varepsilon} P_q^{-1}
    \ll
    H P_q^{\beta-1+\varepsilon} 
    \ll
    H P_Q^{\beta-1+\varepsilon}
    .
  \end{displaymath}
  Choosing $H = \floor{P_Q^{1-\beta-2\varepsilon}}$
  and
  \begin{displaymath}
    Q
    =
    \floor{
      \frac{(1-\beta-2\varepsilon)\log M}{(2-\beta-2\varepsilon) \log 2}
    }
  \end{displaymath}
  so that
  \begin{displaymath}
    2^{Q} \leq M^{\frac{1-\beta-2\varepsilon}{2-\beta-2\varepsilon}}
  \end{displaymath}
  we have $\alpha_1\lambda_1 \to 0$ so that for sufficiently large $M$
  we will have $\alpha_1\lambda_1 \leq 1/2$ and we can apply
  Lemma~\ref{lem:Kusmin_Landau} and we obtain
  \begin{align*}
    \sum_{1\leq q\leq Q}
    \abs{\sum_{M/2^{q+1}<n \leq M/2^q} e(g(n))}
    \ll
    \frac{M}{\abs{h}(aM)^{\beta-\varepsilon}}
    \sum_{1\leq q\leq Q}
    2^{-q(1-\beta+\varepsilon)}
    \ll \frac{M}{\abs{h}(aM)^{\beta-\varepsilon}}.
  \end{align*}

  Plugging this into \eqref{eq:e_well-distribution} we get
  \begin{align*}
    \abs{\sum_{n\leq M}e_{an+b}}
    \ll
    \frac{M}{H} + \frac{M}{2^Q}\log H
    + a^{-\beta+\varepsilon}M^{1-\beta+\varepsilon}.
  \end{align*}
  Replacing $H$ and $Q$ by their values given above we get
  \begin{displaymath}
    \abs{\sum_{n\leq M}e_{an+b}}
    \ll
    M^{\beta+2\varepsilon}
    M^{\frac{(1-\beta-2\varepsilon)^2}{2-\beta-2\varepsilon}}  \log H
    + a^{-\beta+\varepsilon}M^{1-\beta+\varepsilon}
    \ll
    M^{\frac{1-2\beta\varepsilon}{2-\beta-2\varepsilon}}  \log M
    + M^{1-\beta+\varepsilon}
  \end{displaymath}  
  and we obtain the lemma by taking $\sigma = \min(\frac15,\beta)$.
\end{proof}

\begin{lem}\label{lem:well-distribution-medium-beta}
  Suppose that $\tfrac12<\beta(f)<1$ and $a\ll M^{\frac{2-\beta}{\beta}}$, then
  \[
    \abs{\sum_{n\leq M}e_{an+b}}\ll a^{\beta/2}M^{(1+\beta)/2}
  \]
\end{lem}

\begin{proof}
  Here we clearly have $r=2$. Now an application of Lemma \ref{lem:BKS}
  together with $\lambda$ and $\alpha$ as in \eqref{eq:lambda_and_alpha} yields
  \begin{align*}
    \abs{\sum_{P_q< n \leq 2P_q}e(g(n))}
    &\ll P_q\left[
      \left(\abs{h}a^{\beta+\varepsilon}\right)^{\frac{1}{2}}P_q^{\frac{\beta-2}{2}+\varepsilon}
      +\left(\abs{h}a^{\beta-\varepsilon}\right)^{-\frac{1}{R}}P_q^{-\frac{\beta}{2}+\varepsilon}
      +a^\varepsilon P_q^{-\frac{1}{2}+\varepsilon}
    \right]\\
    &\ll P_q^{\frac{\beta}{2}+\varepsilon}\left[
      \left(\abs{h}a^{\beta+\varepsilon}\right)^{\frac{1}{2}}
      +\left(\abs{h}a^{\beta-\varepsilon}\right)^{-\frac{1}{R}}
      +a^\varepsilon
    \right]\\
    &\ll \left(\abs{h}a^{\beta+\varepsilon}\right)^{\frac{1}{2}}P_q^{\frac{\beta}{2}+\varepsilon}.
  \end{align*}

  Summing over all $P_q=M2^{-q}$ for $1\leq q\leq Q$ we get
  \begin{gather*}
    \abs{\sum_{n\leq M}e(g(n))}
    \ll\sum_{1\leq q\leq Q}
      \left(\abs{h}a^{\beta+\varepsilon}\right)^{\frac{1}{2}}\left(M2^{-q}\right)^{\frac{\beta}{2}+\varepsilon}
    \ll\left(\abs{h}a^{\beta+\varepsilon}\right)^{\frac{1}{2}}M^{\frac{\beta}{2}+\varepsilon}.
  \end{gather*}

  Plugging this into \eqref{eq:e_well-distribution} yields
  \[
    \abs{\sum_{n\leq M}e_{an+b}}
    \ll MH^{-1}+\left(Ha^{\beta+\varepsilon}\right)^{\frac{1}{2}}M^{\frac{\beta}{2}+\varepsilon}.
  \]

  This time we set
  \[
    H=\left\lceil M^{\frac32}(aM)^{-\frac{\beta}2}\right\rceil\geq1
  \]
  and obtain the lemma.
\end{proof}

\begin{lem}\label{lem:well-distribution-large-beta}
  Suppose that $1<\beta(f)$ and $a\ll M^{r/\beta-1}$, then
  \[
    \abs{\sum_{n\leq M}e_{an+b}}\ll a^{\beta/(2R-1)}M^{1-\frac{r-\beta}{2R-1}+\varepsilon},
  \]
  where $R=2^{r-1}$.
\end{lem}

\begin{proof}
  In this case $r\geq2$. Similar to above we apply Lemma \ref{lem:BKS}
  with $\lambda_r$ and $\alpha_r$ as in \eqref{eq:lambda_and_alpha} and get
  \begin{align*}
    \abs{\sum_{P_q<n\leq 2P_q}e(g(n))}
    &\ll P_q\left[
      \left(\abs{h}a^{\beta+\varepsilon}\right)^{\frac{1}{2R-2}}P_q^{\frac{\beta-r}{2R-2}+\varepsilon}
      +\left(\abs{h}a^{\beta-\varepsilon}\right)^{-\frac{1}{R}}P_q^{-\frac{\beta}{R}+\varepsilon}
      +a^\varepsilon P_q^{-\frac{1}{R}+\varepsilon}
    \right]\\
    &\ll P_q^{1-\frac{r-\beta}{2R-2}+\varepsilon}\left[
      \left(\abs{h}a^{\beta+\varepsilon}\right)^{\frac{1}{2R-2}}
      +\left(\abs{h}a^{\beta-\varepsilon}\right)^{-\frac{1}{R}}
      +a^\varepsilon\right]\\
    &\ll \left(\abs{h}a^{\beta+\varepsilon}\right)^{\frac{1}{2R-2}}P_q^{1-\frac{r-\beta}{2R-2}+\varepsilon}
  \end{align*}
  where $R=2^{r-1}$.

  Summing over all $P_q=M2^{-q}$ for $1\leq q\leq Q$,
  we get
  \begin{gather*}
    \abs{\sum_{n\leq M}e(g(n))}
    \ll\sum_{1\leq q\leq Q}\left(\abs{h}a^{\beta+\varepsilon}\right)^{\frac{1}{2R-2}}
      \left(M2^{-q}\right)^{1-\frac{r-\beta}{2R-2}+\varepsilon}
    \ll\left(\abs{h}a^{\beta+\varepsilon}\right)^{\frac{1}{2R-2}}
      M^{1-\frac{r-\beta}{2R-2}+\varepsilon}.
  \end{gather*}

  Plugging into \eqref{eq:e_well-distribution} yields
  \[
    \abs{\sum_{n\leq M}e_{an+b}}
    \ll
    \frac{M}{H}
    +\left(Ha^{\beta+\varepsilon}\right)^{\frac{1}{2R-2}}
    M^{1-\frac{r-\beta}{2R-2}+\varepsilon}.
  \]
  
  Setting
  \[
    H=\left\lceil (a^\beta M^{\beta-r})^{-\frac{1}{2R-1}}\right\rceil\geq1
  \]
  yields the lemma in this case.
\end{proof}

For the rest of the proof of Theorem \ref{thm:well-distribution} we need to
specify, what we mean by ``small'' $a$. Therefore we consider the cases
from above again.

\begin{itemize}
\item \textbf{Case 1:} $\beta(f)\leq\tfrac12$. In this case we may always apply
Lemma \ref{lem:well-distribution-small-beta} and the theorem follows.
\item \textbf{Case 2:} $\tfrac12<\beta(f)$. On the one hand if $a\ll
  x^{1-\frac{\beta}{r}}$, then either Lemma
  \ref{lem:well-distribution-medium-beta} or Lemma
  \ref{lem:well-distribution-large-beta} proves the theorem. On the other hand
  if $a\gg x^{1-\frac{\beta}{r}}$, then the trivial bound yields
  \[
    \abs{\sum_{an+b\leq x}e_{an+b}}\ll x^{\beta/r}.
  \]
  For $r\geq2$ we have $r\leq 2R-1$. Therefore
  $\frac{r}{2R-1}\left(1-\frac{\beta}{r}\right)\leq 1-\frac{\beta}{r}$. Thus
  $\beta/r\leq 1-(r-\beta)/(2R-1)$ and the theorem follows
  also in this case.
\end{itemize}

\section{Correlation measure}

Let $f\in\mathcal{B}$ be of type $x^{\ell+}$ with $\ell\geq1$, $\beta=\beta(f)$.
Let $M$ be a positive integer and $\mathbf{d}=(d_1,\ldots,d_s)\in\N^s$ be such
that
\begin{equation}
  \label{eq:condition-d1-ds}
  0\leq d_1<\cdots<d_s\leq N-M.
\end{equation}
Then we need to estimate the correlation sum
\[
  \sum_{n\leq M}e_{n+d_1}\cdots e_{n+d_s}.
\]
Without loss of generality we may assume that $M\geq N^{\frac{9}{10}}$.

For a given $n\leq M$ we have
\[
  e_{n+d_1}\cdots e_{n+d_s}=\sum_{\mathbf{a}\in\{-1,1\}^s}\prod_{j=1}^s
  a_j\mathbf{1}_{\chi(f(n+d_j))=a_j},
\]
where the sum runs over all possible combinations
$\mathbf{a}=(a_1,\ldots,a_s)\in\{-1,1\}^s$. Thus summing over all possible
$n\leq M$ we obtain
\begin{align*}
  \sum_{n\leq M}e_{n+d_1}\cdots e_{n+d_s}
  &=\sum_{\mathbf{a}\in\{-1,1\}^s}\prod_{j=1}^s a_j
    \sum_{n\leq M}\left(\prod_{j=1}^s\mathbf{1}_{\chi(f(n+d_j))=a_j}-2^{-s}\right)
    .
\end{align*}
Writing $\mathbf{x}_n=(f(n+d_1),\ldots,f(n+d_s))$ for $1\leq n\leq M$ we
furthermore get
\[\abs{\sum_{n\leq M}e_{n+d_1}\cdots e_{n+d_s}}
  \leq 2^sMD_M(\mathbf{x}_1,\ldots,\mathbf{x}_M).\]

Now we use our multidimensional discrepancy estimate from Lemma
\ref{lem:koksma-szusz} and get for $H=\lfloor M^{\frac1{10}}\rfloor$
\begin{gather}\label{eq:correlation_koksma-szusz}
  \abs{\sum_{n\leq M}e_{n+d_1}\cdots e_{n+d_s}}
  \ll M^{1-\frac1{10}}+\sum_{0<\varphi(\mathbf{h})\leq H}\frac{1}{\mathbf{r}(\mathbf{h})}
  \abs{\sum_{n\leq M} e\left(h_1f(n+d_1)+\cdots+h_sf(n+d_s)\right)}.
\end{gather}

We concentrate on the exponential sum. For fixed
$\mathbf{0}\neq\mathbf{h}=(h_1,\ldots,h_s)\in[-H,H]^s\cap\Z^s$ we define
\[
  g(\mathbf{h},n)=\sum_{j=1}^{s}h_j f(n+d_j).
\]

By Taylor expansion of $f(n+d)$ around $d=0$ we have
\[
  f(n+d)=f(n)+d
  f'(n)+\frac{d^2}{2}f''(n)+\cdots+\frac{d^{m}}{m!}f^{(m)}(n)+\mathcal{O}\left(d^{m+1}
  f^{(m+1)}(n)\right).
\]
Therefore
\[
  g(\mathbf{h},n)=\sum_{k=0}^m
  \sum_{j=0}^{s}\frac{h_jd_j^k}{k!}f^{(k)}(n)+
  \mathcal{O}\left(\sum_{j=1}^s
  h_jd_j^{m+1}f^{(m+1)}(n)\right).
\]

Now we want to show that the coefficients of $g(\mathbf{h},n)$ cannot be all
zero. For $\mathbf{0}\neq\mathbf{h}=(h_1,\ldots,h_s)\in[-H,H]^s\cap\Z^s$ and
$k\geq0$ an integer we define
\begin{displaymath}
  L_k(\mathbf{h})=h_1d_1^k+\cdots+h_sd_s^k.
\end{displaymath}
Clearly the dominant term of $g(\mathbf{h},n)$ is the smallest $k$
such hat $L_{k}(\mathbf{h})\neq0$. We show that
there always exists such a $k\in\{0,\ldots,s-1\}$.
Suppose on the contrary that
\[
  L_0(\mathbf{h})=\cdots=L_{s-1}(\mathbf{h})=0.
\]
This is a $s\times s$ linear system, whose coefficients form the Vandermonde
matrix over $(d_1,\ldots,d_s)$. Since $0\leq d_1<\cdots<d_s$ are all different
the only solution to this system is $\mathbf{h}=\mathbf{0}$, which we excluded
($\varphi(\mathbf{h})>0$). Therefore this yields a contradiction and for a given
$\mathbf{0}\neq\mathbf{h}\in[-H,H]^s\cap\Z^s$ we set
\begin{displaymath}
  u(\mathbf{h}):=\min\{k\geq0\colon L_k(\mathbf{h})\neq 0\}\leq s-1
  .
\end{displaymath}
Furthermore we collect all $\mathbf{h}$ with $u(\mathbf{h})=u$ in a set
$\mathcal{H}_u$, \textit{i.e.}
\[
  \mathcal{H}_u:=\left\{\mathbf{0}\neq \mathbf{h}\in[-H,H]^s\cap\Z^s\colon u(\mathbf{h})=u\right\}.
\]
In view of \eqref{eq:correlation_koksma-szusz} this yields
\begin{gather}\label{eq:correlation_koksma-szusz-2}
  \abs{\sum_{n\leq M}e_{n+d_1}\cdots e_{n+d_s}}
  \ll M^{1-\frac1{10}}+\sum_{u=0}^{s-1}S_u, 
\end{gather}
where
\[
  S_u:=\sum_{\mathbf{h}\in\mathcal{H}_u}\frac{1}{\mathbf{r}(\mathbf{h})}\abs{\sum_{n\leq M} e\left(g(\mathbf{h},n)\right)}.
\]

We focus on the sum $S_u$ and therefore fix $u\in\{0,1,\ldots,s-1\}$ for the
moment. For $\mathbf{h}\in\mathcal{H}_u$ we have 
\[
  g(\mathbf{h},n)=\sum_{\ell=u}^m
  \frac{L_\ell(\mathbf{h})}{\ell!}f^{(\ell)}(n)+
  \mathcal{O}\left(L_{m+1}(\mathbf{h})f^{(m+1)}(n)\right).
\]
Trivially estimating the terms for $n < c_u \frac{M}{H}$ in the inner
sum of $S_u$ yields
\begin{gather}\label{eq:S_u-1}
  \abs{S_u}\leq \sum_{\mathbf{h}\in\mathcal{H}_u}\frac{1}{\mathbf{r}(\mathbf{h})}\abs{\sum_{c_u\frac{M}{H}\leq n\leq M} e\left(g(\mathbf{h},n)\right)}+\mathcal{O}\left(M^{1-\frac1{10}}(\log M)^s\right),
\end{gather}
where $c_u=\beta(\beta-1)\cdots(\beta-u+1)$. Now we split the
remaining sum over $n$ into $Q=\lfloor \log(H/2c_u)/\log 2\rfloor$
intervals of the form $]M2^{-q},M2^{-q+1}]$ and denote by
$]P_q,2P_q]$ a typical one of them. Again we may suppose that $P_q\geq
M^{\frac{9}{10}}$ and trivially estimating the sum otherwise.

Recall that by Lemma \ref{lem:higher_derivative_order} for integers
$\ell\geq1$ we have
\[
  \frac{f(x)}{x^\ell\log^2(x)}\ll\abs{f^{(\ell)}(x)}
  \ll
  \frac{f(x)}{x^\ell}.
\]
For $\beta=\inf\{c\in[0,\infty[\colon f(x)\prec x^c\}$ and arbitrary
$\varepsilon>0$ we obtain
\[
  x^{\beta-\varepsilon}\ll\abs{f(x)}\ll x^{\beta+\varepsilon}.
\]
Putting these two bounds together we get for any integer $\ell\geq 0$
\[
  x^{\beta-\ell-\varepsilon}\ll\abs{f^{(\ell)}(x)}
  \ll x^{\beta-\ell+\varepsilon}
  .
\]
Then for $r\geq0$ and $x\in]P_q,2P_q]$ we have
\[\lambda_r\ll\abs{\frac{\partial^r}{\partial x^r} g(\mathbf{h},x)}
\ll \alpha_r\lambda_r\]
with
\[
  \lambda_r=\abs{L_u(\mathbf{h})}P_q^{\beta-u-r-\varepsilon}
  \quad\text{and}\quad
  \alpha_r=P_q^{2\varepsilon}.
\]

We pick $w\in \R$ such that $\abs{L_u(\mathbf{h})}=P_q^w$. Since
$0<d_1<d_2<\cdots<d_s\leq N-M$ and $\abs{h_j}\leq H$ we have
\[
  1\leq \abs{L_u(\mathbf{h})}=\abs{\sum_{j=1}^sh_jd_j^u}
  \leq s H N^u\leq s M^{\frac{10u}{9}+\frac1{10}}\leq sP_q^{\frac{100u}{81}+\frac19}.
\]
Therefore $0\leq w\leq c_2$, where $c_2$ depends only on $\beta$.

Then we define
\[
  r:=\left\lceil w+\beta+\frac12\right\rceil
\]
and set $t:=\ell-\beta-w\in\left[\frac12,\frac32\right[$. Note that since $w$ is
bounded, we have $2\leq r\leq \ell$, where $\ell$ depends on $\beta$ only. Now an
application of Lemma \ref{lem:BKS} yields
\begin{align*}
  \sum_{P_q<n\leq 2P_q}e\left(g(\mathbf{h},n)\right)
  &\ll P_q\left[
    \left(\abs{L_u(\mathbf{h})}P_q^{\beta-u-r+\varepsilon}\right)^{\frac1{2R-2}}
    +\left(\abs{L_u(\mathbf{h})}P_q^{\beta-u-\varepsilon}\right)^{-\frac1{R}}
    +P_q^{-\frac1{R}+\varepsilon}
  \right]\\
  &\ll P_q^{1-\frac{u+t}{2R-2}+\varepsilon}
    +P_q^{1-\frac{\beta-u+w}{R}+\varepsilon}
    +P_q^{1-\frac{1}{R}+\varepsilon}\\
  &\ll P_q^{1-\frac{u+\frac12}{2L-2}+\varepsilon}
  +P_q^{1-\frac{\beta-u}{L}+\varepsilon}
  +P_q^{1-\frac{1}{L}+\varepsilon},
\end{align*}
where $R=2^{r-1}$ and $L=2^{\ell-1}$, respectively. Note that it is here, where we need that $\beta>s-1\geq u$.

Plugging this into \eqref{eq:S_u-1} and summing over all
intervals $]M2^{-q},M2^{-q+1}]$ we get
\begin{align*}
  S_u
  &\ll M^{1-\frac{u+\frac12}{2L-2}+\varepsilon}
  +M^{1-\frac{\beta-u}{L}+\varepsilon}
  +M^{1-\frac{1}{L}+\varepsilon}
  +M^{1-\frac1{10}+\varepsilon}.
\end{align*}
Inserting this estimate into \eqref{eq:correlation_koksma-szusz-2} yields
Theorem \ref{thm:correlation_measure_bound}.



\section{Lower bounds}


For this counterexample we assume that $0<c<1$ and $f(x)=x^c$. We will consider
the correlation of order $2$:
\[C_2(N)=\max_{M,d_1,d_2}\abs{\sum_{n=1}^M e_{n+d_1}e_{n+d_2}},\]
where the maximum runs over all $1\leq M\leq N$ and $0\leq d_1<d_2\leq N-M$.

Without loss of generality we set $d_1=0<d_2=1$ and $M=N-1$.
Writing
\begin{multline*}
  \chi\left(n^c\right) \chi\left((n+1)^c\right)
  -
  A_H\left(n^c\right) A_H\left((n+1)^c\right)
  \\
  =
  \chi\left(n^c\right)
  \left( \chi\left((n+1)^c\right) - A_H\left((n+1)^c\right) \right)
  +
  \left( \chi \left(n^c\right) - A_H \left(n^c\right) \right)
  \chi\left((n+1)^c\right)
  \\
  -
  \left( \chi \left(n^c\right) - A_H \left(n^c\right) \right)
  \left( \chi\left((n+1)^c\right) - A_H\left((n+1)^c\right) \right)
  ,
\end{multline*}
since $\abs{\chi}=1$, by Lemma~\ref{lem:vaaler}
we get
\begin{align*}
  \abs{
  \chi\left(n^c\right) \chi\left((n+1)^c\right)
  -
  A_H\left(n^c\right) A_H\left((n+1)^c\right)
  }
  \leq
  B_H \left(n^c\right)
  +
  B_H\left((n+1)^c\right)
  +
  B_H \left(n^c\right) B_H\left((n+1)^c\right)
\end{align*}
and by the definition of $e_n$ we get
\begin{align*}
  \sum_{n=1}^Me_{n}e_{n+1}
  &=\sum_{n=1}^M\chi\left(n^c\right)\chi\left((n+1)^c\right)\\
  &\geq \sum_{n=1}^M A_H\left(n^c\right)A_H\left((n+1)^c\right)
  -\sum_{n=1}^M B_H\left((n+1)^c\right)\\
  &\quad-\sum_{n=1}^M B_H\left(n^c\right)
  -\sum_{n=1}^M B_H\left(n^c\right)B_H\left((n+1)^c\right)
\end{align*}
In the sequel we will provide lower and upper bounds for the products
respectively. To this end let $H\geq1$ be an integer we choose later. By
definition we get for the first product
\begin{align}
  \label{eq:sum-AA}
  &\sum_{n=1}^M A_H\left(n^c\right)A_H\left((n+1)^c\right)\\
  \nonumber
  &\quad=
    \sum_{1\leq \abs{h_1}\leq H}
    \sum_{1\leq \abs{h_2}\leq H} a_{h_1}(H)a_{h_2}(H)
    \sum_{n=1}^M
    e\left(h_1n^c+h_2(n+1)^c\right)\\
    \nonumber
  &\quad=\sum_{1\leq \abs{h_1}\leq H}
  \sum_{\substack{1\leq \abs{h_2}\leq H\\ h_1+h_2\neq 0}} a_{h_1}(H)a_{h_2}(H)
  \sum_{n=1}^M e\left(h_1n^c+h_2(n+1)^c\right)\\
  \nonumber
  &\quad\quad+\sum_{0<\abs{h}\leq H}a_{h}(H)a_{-h}(H)
    \sum_{n=1}^M e\left(h\left((n+1)^c-n^c\right)\right).
\end{align}

We concentrate on the first inner exponential sum. We assume slightly
more generally that $0\leq \abs{h_1} \leq H$, $0\leq \abs{h_2} \leq H$
with $h_1+h_2\neq 0$.
As above we dyadically split
the sum up into sub-sums over intervals $]2^{-q-1}M,2^{-q}M]$
for $0\leq q\leq Q$ for some $Q$ to be chosen later (with $2^Q\leq M$):
\begin{displaymath}
  \sum_{n=1}^M   e\left(h_1n^c+h_2(n+1)^c\right)
  =
  \sum_{1\leq n \leq P_Q}   e\left(h_1n^c+h_2(n+1)^c\right)
  +
  \sum_{q=0}^Q
  \sum_{P_q < n \leq 2 P_q}
  e\left(h_1n^c+h_2(n+1)^c\right)
  ,
\end{displaymath}
where $P_q=M/2^{q+1}$.
The first sum can be trivially bounded by $\frac{M}{2^{Q+1}}$.
Moreover we set
\[
  g(x)=h_1 x^c + h_2 (x+1)^c
\]
for short. Then in the case of $h_1+h_2\neq 0$,
for $0\leq q\leq Q$ and $\frac{M}{2^{q+1}} < x \leq \frac{M}{2^{q}}$
we have
\[
  c (2^{-q}M)^{c-1}
  \ll
  \abs{g'(x)}=\abs{h_1cx^{c-1}+h_2c(x+1)^{c-1}}
  \leq
  2 H c (2^{-q-1}M)^{c-1}
  \leq
  \frac12
\]
provided that we choose $H$ and $Q$ such that
\begin{displaymath}
  H=\floor{\frac{M^{(1-c)^2}}{4}}
\end{displaymath}
\begin{displaymath}
  Q = \floor{ \frac{\log\left(M (4 H)^{-1/(1-c)}\right)}{\log 2}}
  .
\end{displaymath}
An application of Lemma \ref{lem:Kusmin_Landau} yields
\begin{align*}
  \sum_{q=0}^Q
  \abs{
  \sum_{P_q < n \leq 2 P_q}
  e\left(h_1n^c+h_2(n+1)^c\right)}
  \ll_c \sum_{q=0}^Q 2^{q(c-1)}M^{1-c}
  \ll_c M^{1-c},
\end{align*}
where we have used that $c-1<0$.
It follows that for $0\leq \abs{h_1} \leq H$,
$0\leq \abs{h_2} \leq H$ with $h_1+h_2\neq 0$ we have
\begin{equation}
  \label{eq:exponential-sum-estimate-h1-h2}
  \abs{
    \sum_{n=1}^M   e\left(h_1n^c+h_2(n+1)^c\right)
  }
  \ll_c M^{1-c} + \frac{M}{2^{Q+1}}
  \ll_c  M^{1-c} + (4H)^{1/(1-c)}
  \ll_c  M^{1-c} 
  .
\end{equation}
Plugging this into the sum over all
$h_1+h_2\neq 0$ we get
\begin{multline*}
  \abs{
    \sum_{\substack{0<\norm{\mathbf{h}}_\infty\leq H\\ h_1+h_2\neq 0}}a_{h_1}(H)a_{h_2}(H)
    \sum_{n=1}^M e\left(h_1n^c+h_2(n+1)^c\right)
  }\\
  \ll M^{1-c}\left(
    \sum_{\substack{0<\norm{\mathbf{h}}_\infty\leq H\\ h_1+h_2\neq 0\\ h_1\neq0, h_2\neq0}}\frac{1}{\abs{h_1h_2}}+
    \sum_{0<h\leq H}\frac{1}{h}
  \right)
  \ll M^{1-c}(\log H)^2,
\end{multline*}
where we have used that $a_h(H)\asymp \frac1{\abs{h}}$
for $1\leq \abs{h}\leq H$.

Now we concentrate on the sum over those pairs $(h_1,h_2)$ with $h_1+h_2=0$.
As above we divide the exponential sum into dyadic intervals:
\[
  \sum_{n=1}^M e\left(h\left((n+1)^c-n^c\right)\right)
  =
  \sum_{0\leq q\leq Q}
  \sum_{P_q < n \leq 2 P_q}
  e\left(h\left((n+1)^c-n^c\right)\right).
\]
Expanding $(x+d)^c$ around $d=0$ yields
\[
  (x+d)^c=x^c+cdx^{c-1}+\mathcal{O}\left(d^2x^{c-2}\right).
\]
Thus
\[
  g(n)=h((n+1)^c-n^c)=chn^{c-1}+\mathcal{O}\left(h2^{(q+1)(2-c)}M^{c-2}\right)\quad\text{for }n\in[2^{-q-1}M,2^{-q}M].
\]
We have
\begin{align*}
  \MoveEqLeft
  \abs{
  \sum_{0\leq q\leq Q}
    \sum_{P_q < n \leq 2 P_q}
    e\left(h\left((n+1)^c-n^c\right)\right)
  }\\
  &=
    \abs{
    \sum_{0\leq q\leq Q}
    \sum_{P_q < n \leq 2 P_q}
    e\left(chn^{c-1}\right)
    }
    +\mathcal{O}\left(\sum_{0\leq q\leq Q}\abs{h}2^{(q+1)(1-c)}M^{c-1}\right)\\
  &=
    \abs{\sum_{P_Q<n\leq M} e\left(chn^{c-1}\right)}
    +\mathcal{O}\left(\abs{h}\right)
    =
    \abs{\sum_{n=1}^{M} e\left(chn^{c-1}\right)}
    +\mathcal{O}\left(\frac{M}{2^{Q+1}}\right)
    +\mathcal{O}\left(\abs{h}\right)
    .
\end{align*}
Together with an application of Lemma \ref{lem:robert} we get
\begin{align*}
  \abs{
  \sum_{0\leq q\leq Q}
    \sum_{P_q < n \leq 2 P_q}
    e\left(h\left((n+1)^c-n^c\right)\right)
  }
  \geq \frac{M}{8}+\mathcal{O}\left(\abs{h}\right).  
\end{align*}
Plugging this into \eqref{eq:sum-AA} we get
\begin{multline}
  \label{eq:sum-AA-lowerbound}
  \sum_{n=1}^M A_H\left(n^c\right)A_H\left((n+1)^c\right)\\
  \geq
  \frac{M}8 \sum_{0<h\leq H}a_{h}(H)a_{-h}(H)
  +\mathcal{O}\left(\sum_{0<h\leq H}\frac1{h}\right)
  +\mathcal{O}\left(M^{1-c}(\log H)^2\right)
  \gg M.
\end{multline}

Since $b_h(H)\asymp \frac1{H}$ for $0<\abs{h}\leq H$ we get along
similar lines using \eqref{eq:exponential-sum-estimate-h1-h2}
with $h_2=0$ and the definition of $B_H$ in Lemma~\ref{lem:vaaler}
that

\begin{align*}
  \sum_{n=1}^MB_H(n^c)
  &\ll \frac{M}{H} + M^{1-c},\\
  \sum_{n=1}^MB_H((n+1)^c)
  &\ll \frac{M}{H} + M^{1-c},\\
  \sum_{n=1}^MB_H(n^c)B_H((n+1)^c)
  &\ll
    \frac{M}{H^2} + 2 \frac{M}{H} + 2 M^{1-c}
    +
    \sum_{h_1+h_2=0} \frac{M}{H^2}
    +
    \sum_{h_1+h_2\neq 0} \frac{M^{1-c}}{H^2}
  \\
  &\ll
    \frac{M}{H} +  M^{1-c}
    .
\end{align*}

Combining this last upper bound with the lower
bound~\eqref{eq:sum-AA-lowerbound} we 
finally get that
\[
  C_2(N)\geq\sum_{n=1}^{N-1}e_{n}e_{n+1}\gg N.
\]

\section*{Acknowledgements}
The first author is supported by project ANR-18-CE40-0018 funded by the French
National Research Agency. The third author acknowledges support by project
I~4406-N funded by the Austrian Science Fund.

Major parts of this article were established when the three authors profited
from a research in residence stay at the Centre International de Rencontres
Mathématiques (CIRM), Marseille, France. The authors
thank the institution for its hospitality.


\end{document}